\documentclass[12pt,reqno]{amsart}
\usepackage{amsfonts,amssymb,amsmath}
\usepackage{datetime}
\usepackage{xcolor}

\definecolor{fgreen}{RGB}{44,144, 14}

\renewenvironment{proof}{{\bfseries Proof.}}{\qed}

\numberwithin{equation}{section} 
\newtheorem*{thm}{Theorem}
\newtheorem{theorem}{Theorem}[section] 
\newtheorem{proposition}[theorem]{Proposition} 
\newtheorem{corollary}[theorem]{Corollary} 
\newtheorem{lemma}[theorem]{Lemma} 

\theoremstyle{definition}
\newtheorem{definition}[theorem]{Definition} 
\newtheorem{remark}[theorem]{Remark} 
\newtheorem{example}[theorem]{Example}
\usepackage[colorlinks=true,backref,pagebackref,citecolor=cyan,urlcolor=blue,linkcolor=blue]{hyperref}
\def\+{\oplus}

\def\<>{\langle \cdot\, , \cdot \rangle}

\newcommand{\secref}[1]{Section~\ref{#1}}
\newcommand{\thmref}[1]{Theorem~\ref{#1}}
\newcommand{\lemref}[1]{Lemma~\ref{#1}}
\newcommand{\remref}[1]{Remark~\ref{#1}}
\newcommand{\propref}[1]{Proposition~\ref{#1}}
\newcommand{\corref}[1]{Corollary~\ref{#1}}

\begin{document} 
	
\title[On Commuting automorphisms of nilpotent Lie algebras]{On Commuting automorphisms of nilpotent Lie algebras}
 \author{Shushma Rani$^{\ast}$}
\address{Harish Chandra Research Institute, Chhatnag Road, Jhunsi, Prayagraj, 211019, Uttar Pradesh, India}
\email{shushmarani95@gmail.com, shushmarani@hri.res.in}

\author{Niranjan Nehra}
	\address{Indian Institute of Science Education and Research Mohali, Knowledge City,  Sector 81, S.A.S. Nagar 140306, Punjab, India}
	\email{niranjannehra11@gmail.com, ph17042@iisermohali.ac.in}
 
 \author{Rohit Garg}
 \address{Department of Mathematics, Government Ripudaman College, Nabha, Punjab 147201, India; School of Mathematical Sciences, National Institute of Science Education and Research, Bhubaneswar, HBNI, P.O. Jatni, Khurda, Odisha 752050, India.}
%\address{National Institute of Science Education and Research Bhubaneswar P.O. Jatni, Khurda 752050, Odisha, India}
\email{rohitgarg289@gmail.com, rohitgarg289@niser.ac.in}
%\author{Tushar Kanta Naik}
%\author{Tanusree Khandai}
%	\address{Indian Institute of Science Education and Research Mohali, Knowledge City,  Sector 81, S.A.S. Nagar 140306, Punjab, India}
%	\email{tanusree@iisermohali.ac.in}
\thanks{$^{*}$-The corresponding author.}

\begin{abstract}
We investigate the commuting automorphisms of nilpotent Lie algebras $L$ with coclass $\leq 3$. Our examination exposes the conditions under which the set of commuting automorphisms of $L$ forms a subgroup within its automorphism group.
\end{abstract}

\maketitle
 
%\tableofcontents	

%%%%%%%%%%%%%%%%%%%%%%%%%%%%%%%%%%%%%%%%%%%%%%%%%%%%%%%%%%%%%%%%%%%%%%%%%%%%%%%%%%%%%%%%%%%%%%%%%%%%%%%%%%%%%%%%%%%%%%%%%%%%%%%%%%%%%%%%%%%%%%%%%%%%%%%%%%%%%%%%%%%%%%%%%%%%%%%%%%%%%%%%%%%%%%%%%%%%%%%%%%%%%%%%%%%%%%%%%%%%%%%%%%%%%%%%%%%%%%%%%%%%%%%%%%%%%%%%%%%%%%%%%%%%%%%%%
\section{Introduction} 
The exploration of commuting automorphisms within algebraic structures has garnered significant attention. Our study focuses on finite-dimensional nilpotent Lie algebras $L$ over field $\mathbb{F}$ of characteristic different from $2$. An automorphism $ \alpha $ in $ Aut(L) $ is termed commuting if $ [\alpha(x), x] = 0 $ for all $ x \in L $, and central if $ \alpha(x) - x \in Z(L) $, where $ Z(L) $ denotes the center of $ L $. The sets of commuting automorphisms and central automorphisms of $ L $ are denoted as $ \mathcal{A}(L) $ and $ Aut_c(L) $, respectively. While $Aut_c(L) \subseteq \mathcal{A}(L)$ forms a subgroup of $Aut(L)$, $\mathcal{A}(L)$ itself may not be a subgroup.

The investigation of commuting automorphisms finds its roots in ring theory \cite{ring1, ring2}, before extending into group theory with a seminal question posed by Herstein. Herstein conjectured whether the identity remains the only commuting automorphism within a simple non-abelian group $ G $. Building upon Herstein's inquiry, subsequent researchers, including Laffey and Pettet \cite{1986}, further expanded the investigation. Pettet's work established conditions under which $ \mathcal{A}(G) $ reduces to the identity automorphism. 

Inspired by research on commuting automorphisms of groups, the exploration of commuting automorphisms for Lie algebras commenced in \cite{commutLie}, where the author characterized Lie algebras for which $ \mathcal{A}(L) = 1_L $, along with results similar to those obtained for groups in \cite{2001,2002}. Additionally, in \cite{2002}, the author raised natural questions regarding the nature of the set of commuting automorphisms $ \mathcal{A}(G) $, particularly whether it forms a subgroup of the group $ Aut(G) $, and under what conditions $ \mathcal{A}(G) $ equals the group of central automorphisms $ Aut_c(G) $. Subsequent research explored this question for finite $p$-groups of coclass $1,2$, and $3$ \cite{2001, 2002, 2013,2015,2016,2013a,2024, rohit2021, rohit2022, rohit2024}, delineating cases where $\mathcal{A}(G)= Aut_c(G)$ and where $\mathcal{A}(G)$ indeed forms a group. For some other classes of finite $p$-groups $G$ for which the set of commuting automorphisms of $G$ forms a subgroup of $Aut(G)$, the readers are advised to refer \cite{rohit2019, rohit2020}.

Our study extends this inquiry to nilpotent Lie algebras of coclass $ 1, 2 $, and $ 3 $ which are similar to the results in p-groups \cite{2013a,2015,2016,2024}. The main contributions of this paper are as follows:

\begin{thm}
    Let $L$ be a  Heisenberg Lie algebra of dimension $n$ ($ n \geq 5$), over filed $\mathbb F$ and $[L: Z(L)] \geq 4$, then $\mathcal{A}(L)$ does not form a subgroup of $Aut(L)$.
 \end{thm} 

\begin{thm} Let $L$ be $n$-dimensional, nilpotent Lie algebra with coclass $1$ over field $\mathbb F$ of characteristic different from $2$. Then $\mathcal{A}(L)= Aut_c(L)$.   
\end{thm}
\begin{thm} Let $L$ be $n$-dimensional, nilpotent Lie algebra with coclass $2$ over field $\mathbb F$ of characteristic different from $2$. Then $\mathcal{A}(L)$ forms a subgroup of $Aut(L)$.   
\end{thm}
%%%%%%%%
\iffalse
\begin{thm} Let $L$ be $n$-dimensional, nilpotent Lie algebra with coclass $3$ over field $\mathbb F$ of characteristic different from $2$. If one of the followings hold:   
\begin{enumerate}
    \item  $n \leq 4$
    \item $n=5$ and $\dim Z(L) \neq 1$
    \item  for $n \geq 6$,
   \begin{itemize}
        \item[(i)]  nilpotency class of $Z_2(L) \neq 2$
       \item[(ii)] if nilpotency class of $Z_2(L)$ is $2$ then $\dim Z(L) \neq 1,$ or $ \dim Z_2(L)\neq 4$ or $\dim L'\neq n-4$ 
   \end{itemize}
   \end{enumerate}
    then $\mathcal{A}(L)$ forms a subgroup of $Aut(L)$. 
\end{thm}
\fi
%%%%%%%%%%%

\begin{thm}
    For an $n$-dimensional nilpotent Lie algebra $L$ over field $\mathbb F$ with characteristic different from $2$ and coclass $3$, the following conditions imply that $\mathcal{A}(L)$ forms a subgroup of $Aut(L)$: \begin{enumerate}
\item The dimension of $L$ is at most $4$.
\item The dimension of $L$ is 5, and the center of $L$ has a dimension different from $1$.
\item The dimension of $L$ is at least $6$, and one of the following conditions is satisfied:
\begin{itemize}
\item The nilpotency class of the second center $Z_2(L)$ is not equal to $2$.
\item The nilpotency class of $Z_2(L)$ is $2$, and at least one of the following is true:
\begin{itemize}
\item The dimension of the center of $L$ is not equal to $1$.
\item The dimension of the second center $Z_2(L)$ is not equal to $4$.
\item The dimension of the derived subalgebra $L'$ is not equal to $n-4$.
\end{itemize}
\end{itemize}
\end{enumerate} In other words, if any of these conditions hold, then the set of commuting automorphisms of $L$ form a subgroup of the full automorphism group of $L$.
\end{thm}

The paper is structured as follows: \secref{sec-2} provides a review of necessary background, and previously established results on commuting automorphisms from \cite{commutLie}. \secref{sec-3} introduces essential lemmas crucial for proving the main theorems and establishing the theorem where $\mathcal{A}(L)$ does not form a subgroup of $Aut(L)$ for Heisenberg Lie algebras. Other main theorems are then demonstrated in subsequent sections:  \secref{sec_coclass1} establishes the equality $Aut_c(L)=\mathcal{A}(L)$ for nilpotent Lie algebras of coclass $1$. \secref{sec_coclass2} and \secref{sec_coclass3} respectively deal with the cases of coclass $2$ and certain cases of coclass $3$, where $\mathcal{A}(L)$ forms a subgroup of $Aut(L)$.

%%%%%%%%%%%%%%%%%%%%%%%%%%%%
\subsection*{\em{Acknowledgements}}{\em The authors express gratitude to Dr. Tushar Kanta Naik for introducing the problem. Additionally, the first author acknowledges the support received from the HRI postdoctoral fellowship.}

%%%%%%%%%%%%%%%%%%%%%%%%%%%%%%%%%%%%%%%%%%%%%%%%%%%%%%%%%%%%%%%%%%%%%%%%%%%%%%%%%%%%%%%%%%%%%%%%%%%%%%%%%%%%%%%%%%%%%%%%%%%%%%%%%%%%%%%%%%%%%%%%%%%%%%%%%%%%%%%%%%%%%%%%%%%%%%%%%%%%%%%%%%%%%%%%%%%%%%%%%%%%%%%%%%%%%%%%%%%%%%%%%%%%%%%%%%%%%%%%%%%%%%%%%%%%%%%%%%%%%%%%%%%%%%%%%
\section{Preliminaries}\label{sec-2}
Throughout the paper, $L$ denotes a Lie algebra of finite dimension over an arbitrary field $\mathbb F$. 

\begin{definition}
    Let $L$ be a Lie algebra of an arbitrary dimension of an arbitrary field $\mathbb F$ and $Z(L)$, $Z_2(L)$ denotes the center, and the second center of the Lie algebra, respectively, defined as
    $$Z(L)= \{ x \in L: [x,y]=0 \ \forall \ y \in L\}.$$
    $$ Z_2(L)= \{ x \in L: [x, L] \subseteq Z(L)\}.$$
    An automorphism $\alpha: L \longrightarrow L$ on Lie algebra $L$ is a Lie algebra homomorphism that is one-one and onto. The set of all automorphisms of $L$ is denoted by $Aut(L)$, which forms a group under the operation composition of maps.
\end{definition}

\begin{definition}
    An automorphism $\alpha: L \longrightarrow L$ is called central if $\alpha(x)-x \in Z(L)$ for all $x \in L$ and it is called commuting automorphism if $[\alpha(x), x]=0$ for all $x \in L.$ Fix some notations:
    $$Aut_c(L):= \{ \alpha \in Aut(L) : \alpha(x)-x \in Z(L) \ \forall \ x \in L\}$$
    $$ \mathcal{A}(L):= \{ \alpha \in Aut(L) : [\alpha(x), x]=0 \ \forall \ x \in L \}$$
    It is easy to check that $\alpha^{-1} \in \mathcal{A}(L)$ for all $\alpha \in \mathcal{A}(L).$ Observe that $Aut_c(L) \subseteq \mathcal{A}(L).$ $Aut_c(L)$ is a subgroup of $Aut(L)$ but $\mathcal{A}(L)$ need not be a subgroup of $Aut(L)$, see \thmref{heisenberg}. 
\end{definition}
\begin{definition}
    A Lie algebra $L$ is said to be nilpotent Lie algebra if the lower central series of $L$, defined as $$L \supset L' \supset L^2 \supset \cdots L^i \supset \cdots $$ terminates at a finite step, i.e., there exists an $n \in \mathbb Z_+$ such that $L^n=0$ and the least such $n$, i.e., $L^{n-1} \neq 0$ and $L^n=0$, is called the nilpotency class of $L$. If $L$ is finite-dimensional nilpotent Lie algebra, say $\dim L=k$, with nilpotency class, $n$, then $k-n$ is called coclass of $L$. An $n$-dimensional nilpotent Lie algebra of coclass $1$ is said to be a Lie algebra of maximal nilpotency class. In this case, $\dim L'= n-2$ and $[L^i: L^{i+1}]=1$ for all $1 \leq i \leq n-2$ and $L^{n-1}=0.$
\end{definition}
\begin{definition}
    Two Lie algebras $L_1$ and $L_2$ over field $\mathbb F$ are said to be isoclinic, whenever there exist isomorphisms $\eta_1: L_1/Z(L_1) \longrightarrow L_2/Z(L_2)$ and $\eta_2: L_1' \longrightarrow L_2'$ such that if $$ \eta_1(u_i + Z(L_1))= v_i+ Z(L_2), \text{ for } i=1,2, $$ then $\eta_2 [u_1, u_2]= [v_1, v_2]$. The pair $(\eta_1, \eta_2)$ is called isoclinism between $L_1$ and $ L_2$.
\end{definition}
\begin{remark}
    Any finite dimensional nilpotent Lie algebra with $1$ dimensional derived subalgebra is a Heisenberg Lie algebra up to isoclinism. 
\end{remark}

First, we will recall some of the properties of $\mathcal{A}(L)$ from \cite{commutLie}. For the sake of completeness, we are adding proofs here.

\begin{lemma}\label{firstlem}
Let $L$ be a Lie algebra of an arbitrary dimension over field $\mathbb F$ and  $\alpha \in \mathcal{A}(L)$ and $x,y \in L.$ Then 
\begin{itemize}
    \item[(i)] $[\alpha(x), y]=[x, \alpha(y)]$
    \item[(ii)] $[\alpha(x)-x, y]=[x, \alpha(y)-y]$
    \item[(iii)] $\alpha(Z(L)) \subseteq Z(L).$
\end{itemize}    
\end{lemma}
\begin{proof}
    Since $\alpha \in \mathcal{A}(L)$ and $x, y \in L$, $[\alpha(x+y), x+y]=0 \implies [\alpha(x), y]+[\alpha(y), x]=0 .$ Thus (i) follows. Parts (ii) and (iii) follow easily from this.
\end{proof}

%%%%%%%%%%%%%%%%%
\iffalse
\begin{lemma}\cite[Theorem 3.2]{commutLie} Let $\alpha, \beta \in \mathcal{A}(L)$. Then
\begin{itemize}
    \item[(i)] $\alpha^n \in \mathcal{A}(L)$ for all $n \in \mathbb Z.$
    \item[(ii)] if $\alpha\beta \in \mathcal{A}(L)$ then $[\alpha, \beta] \in Aut_c(L).$
    \item[(iii)] if $[\alpha, \beta] \in Aut_c(L)$ and $char \, \mathbb F \neq 2$ then  $\alpha\beta \in \mathcal{A}(L).$
\end{itemize}    
\end{lemma}

\begin{remark}
    Let $char \ \mathbb F \neq 2$. Then $\mathcal{A}(L)$ is a subgroup of $Aut(L)$ if and only if $[\alpha, \beta] \in Aut_c(L)$
\end{remark}
\fi
%%%%%%%%%%%%%%%%%%%

\begin{lemma}\label{shush1a} Let $L$ be a Lie algebra of an arbitrary dimension over field $\mathbb F$ and  $\alpha \in \mathcal{A}(L)$ then
    \begin{itemize}
        \item[(i)] $[\alpha(x)-x,[y,z]]=[\alpha(y)-y, [x,z]]$
 for all $x,y,z \in L$.
 \item [(ii)] $[y,[y, \alpha(x)-x]=0$ for all $x, y \in L$
 \item[(iii)] $[\alpha(x)-x, [y,z]]= 2[z,[y, \alpha(x)-x]]$  for all $x,y,z \in L$.
 \item[(iv)] $[\alpha(x)-x, [y,z]]=0$ for all $x,y,z \in L$.
 \end{itemize}
\end{lemma}

\begin{proof} Let $x, y, z \in L$
\begin{enumerate}
    \item[(i)] Using Jacobi identity and \lemref{firstlem}, we get,
    $$\begin{aligned} [x, [\alpha(z), y]]+[y, [x, \alpha(z)]] &=  [\alpha(z), [x,y]]=[z, \alpha([x,y])]=[z, [\alpha(x), \alpha(y)]]\\ & = [\alpha(x), [z, \alpha(y)]]+[\alpha(y), [\alpha(x), z]]\\  &= [\alpha(x), [\alpha(z), y]]+ [\alpha(y), [x, \alpha(z)]]\\ 
    [\alpha(x)-x,[y,\alpha(z)]] &=[\alpha(y)-y, [x,\alpha(z)]] 
    \end{aligned}$$\\ 
 Hence (i) follows as $\alpha$ is an onto map.
  \item[(ii)] Using Jacobi identity and \lemref{firstlem}, we get,
  $$\begin{aligned}[y,[y, \alpha(x)] &= [y, [\alpha(y), x]]= [y, \alpha([y, \alpha^{-1}(x)])]=[\alpha(y), [y, \alpha^{-1}(x)]]\\ &=[y, [\alpha(y), \alpha^{-1}(x)]]+[\alpha^{-1}(x), [y, \alpha(y)]]= [y, [\alpha(y), \alpha^{-1}(x)]]=[y,[y,x]] \end{aligned}$$
Hence, $[y,[y, \alpha(x)-x] =0 .$

\item[(iii)] Using part(ii), we get, $$\begin{aligned} 0 &= [y+z, [y+z, \alpha(x)-x]=[y,[z, \alpha(x)-x]]+[z, [y, \alpha(x)-x]]\\  
[z, [y, \alpha(x)-x]] &=-[y,[z, \alpha(x)-x]] \\ [\alpha(x)-x, [y,z]] &= -[y,[z, \alpha(x)-x]]-[z,[\alpha(x)-x, y]] \quad  \text{using Jacobi identity}\\ &= 2[z, [y, \alpha(x)-x]]
\end{aligned}$$
\item[(iv)] Using part(i), (iii) and \lemref{firstlem},
$$\begin{aligned} [\alpha(x)-x, [y,z]] &= 2[z,[y, \alpha(x)-x]]= 2[z,[\alpha(y)-y,x]] = [\alpha(y)-y, [z,x]] \\ &= - [\alpha(y)-y, [x,z]] = -[\alpha(x)-x, [y,z]]\\ 
[\alpha(x)-x, [y,z]] & =0 \text{ if } char \, \mathbb F \neq 2.\\
[\alpha(x)-x, [y,z]] & =0 \text{ if } char \, \mathbb F=2 \text{ (by part(iii)}
\end{aligned}$$
 \end{enumerate} 
\end{proof}

%\begin{lemma}\label{shush1a}\cite[Theorem 3.6]{commutLie} If $\alpha \in \mathcal{A}(L)$ and $x \in L$ then $\alpha(x)-x \in C_L(L').$   \end{lemma}

\begin{corollary}\label{shush1}If $\alpha \in \mathcal{A}(L)$ and $x \in L$ then $\alpha(x)-x \in Z_2(L).$      
\end{corollary} 
\begin{proof} Let $\alpha \in \mathcal{A}(L)$, then $[[\alpha(x)-x,y],z]=0$ for all $x, y, z \in L$, implies that $[\alpha(x)-x, y] \in Z(L), \forall x, y \in L$, i.e., $\alpha(x)-x \in Z_2(L)$ for all $x \in L$.    
\end{proof}
%%%%%%%%%%%%%%%%%%%%%%%%%%%%%%%%%%%%%%%%%%%%%%%%%%%%%%%%%%%%%%%%%%%%%%%%%%%%%%%%%%%%%%%%%%%%%%%%%%%%%%%%%%%%%%%%%%%%%%%%%%%%%%%%%%%%%%%%%%%%%%%%%%%%%%%%%%%%%%%%%%%%%%%%%%%%%%%%%%%%%%%%%%%%%%%%%%%%%%%%%%%%%%%%%%%%%%%%%%%%%%%%%%%%%%%%%%%%%%%%%%%%%%%%%%%%%%%%%%%%%%%%%%%%%%%%%
\section{Main results}\label{sec-3}
 
 % abelian means $C_L(L')$ i abelian
\begin{lemma}\label{abelian} Let $L$ be a Lie algebra over field $\mathbb F$ such that $char \, \mathbb F \neq 2$ and $Z_2(L)$ is abelian. Then $\mathcal{A}(L)$ is a subgroup of $Aut(L)$.
\end{lemma}
\begin{proof} Let $Z_2(L)$ is abelian and $\alpha, \beta \in \mathcal{A}(L)$. By \corref{shush1}, $\alpha(x)-x, \beta(x)-x \in Z_2(L)$ for all $x\in L.$ Take $\alpha(x)= x+z_1$ and $\beta(x)= x+z_2$ where $z_1, z_2 \in Z_2(L)$.
 $$ \begin{aligned}
     [ z_1, x ] &= [\alpha(x)-x, x]= [\alpha(x), x] =0\\
     [ z_2, x ] &= [\beta(x)-x, x]= [\beta(x), x]=0 \\
     [\alpha(x), \beta(x)] &= [x+z_1, x+z_2]= [z_1, z_2]=0 \quad (\because Z_2(L) \text{ is abelian})\\
     [\beta\alpha(x), x] &= [\alpha(x), \beta(x)],  \text{ by } \lemref{firstlem}\\
     [\beta\alpha(x), x] &=0, \ \forall \ x \in L,
 \end{aligned}$$
Thus $\beta \alpha \in \mathcal{A}(L)$ for all, $\alpha, \beta \in \mathcal{A}(L).$  Hence the lemma.  
\end{proof}

%%%%%%%%%%%%%%%%%%%%%%

% wholeL means $C_L(L')=L$
\begin{lemma}\label{wholeL} Let $L$ be a nilpotent Lie algebra over field $\mathbb F$ such that $char \, \mathbb F \neq 2$ and $L' \subseteq Z(L)$, $[L: Z(L)]=2$. Then $\mathcal{A}(L)$ is a subgroup of $Aut(L)$.
\end{lemma}

\begin{proof}  Since $[L: Z(L)]=2$, we can choose $x,y \in L$ such that $L= \left\langle \ x, y, Z(L)  \ \right\rangle$. Then $ [x,y] \in L' \subseteq Z(L)$. Let $\alpha \in \mathcal{A}(L)$ then
 $\alpha(x)= a_1 x+a_2 y+z_1$ and $\alpha(y)= b_1 x+b_2 y+ z_2$, where $a_i, b_i \in \mathbb F$ for all $i=1,2$, $z_1, z_2 \in Z(L)$.
$$\begin{aligned}
    [\alpha(x), x] &= [a_1 x+ a_2 y+ z_1, x]= a_2 [y,x]\\
    [\alpha(y), y] &= [b_1 x+ b_2 y+ z_2, y]= b_1 [x,y]
\end{aligned}$$
    But $[\alpha(u), u]=0$ for all $u \in L$  and $[x,y] \neq 0$ implies that $a_2=0=b_1$. By \lemref{firstlem}, 
    $$[\alpha(x), y]=[x, \alpha(y)]  \implies a_1[x,y]= b_2[x,y] \implies a_1= b_2= a( \text{say}).$$ Thus, $\alpha(x)= a x+ z_1$ and $\alpha(y)= a y+ z_2$.

    Let $\beta_1, \beta_2 \in \mathcal{A}(L)$ then by above discussion, $\beta_i(x)= c_i x+ z_i$ and $\beta_i(y)= c_i y+ z_i'$ where $c_1, c_2 \in \mathbb F$ and $z_i, z_i' \in Z(L)$ for $i=1,2$.
    $$\begin{aligned}
        [\beta_1(x), \beta_2(x)] &= [c_1 x+ z_1, c_2 x+z_2]=0\\
        [\beta_1(y), \beta_2(y)] &= [c_1 y+ z_1', c_2 y+z_2']=0
    \end{aligned}$$
    Thus $[\beta_2\beta_1(u), u]= [\beta_1(u), \beta_2(u)]=0$ for all $u \in L$, i.e.,  $\beta_2\beta_1 \in \mathcal{A}(L)$ for all $\beta_2, \beta_1 \in \mathcal{A}(L)$. Hence, the lemma.
\end{proof}

%%%%%%%%%%%%%%%%%%%%%%%%%%%%

\begin{lemma}\label{shush2} Let $L$ be a nilpotent Lie algebra over field $\mathbb F$ such that $char \, \mathbb F \neq 2$. Assume that, $[Z_2(L): Z(L)]=2$ and $Z(L)= L^k$ for some $k \geq 1$. Then $\mathcal{A}(L)$ is a subgroup of $Aut(L)$.    
\end{lemma}

\begin{proof} If $k=1$ then $Z(L)=L'$ and $[L: Z(L)]=2$. By \lemref{wholeL}, the result follows. Assume $k \geq 2$. Since $Z(L)=L^k$, $L^{k-1} \subseteq Z_2(L)$. Observe that $L^{k-1} \subseteq Z(Z_2(L))$
% as $k \geq 2,$ $L^{k-1}= [L, L^{k-2}]$. Let $y \in Z_2(L)$ then $[y, L] \subseteq Z(L).$ $[y, [L, L^{k-2}]]=0 $ by jacobi identity.
and $[Z_2(L): L^{k-1}] \leq 1$. Hence, $Z_2(L)$ is abelian. By \lemref{abelian}, the result follows.  
\end{proof}

\medskip
%%%%%%%%%%%%%%%%%%%%%%%%%%%%%%%%%%%%%%%%%%%%%%%%%%%%%%%%%%%%%%%%%%%%%%%%%%%%%%%%%%%%%%%%%%%%%%%%

Typically, $\mathcal{A}(L)$ does not necessarily act as a subgroup of $Aut(L)$. The subsequent theorem illustrates that, in the case of Heisenberg Lie algebra containing a minimum of $4$ non-central elements, $\mathcal{A}(L)$ fails to be a subgroup of $Aut(L)$.

\begin{theorem}\label{heisenberg}
    Let $L$ be a  Heisenberg Lie algebra of dimension $n$ ($ n \geq 5$), over field $\mathbb F$ and $[L: Z(L)] \geq 4$, then $\mathcal{A}(L)$ is not a subgroup of $Aut(L)$.
 \end{theorem}   
\begin{proof}
Since $L$ is a  Heisenberg Lie algebra of dimension $n$, $\dim L'=1$ and $L' \subseteq Z(L)$. Then $L$ has a basis $\{u_1, u_2, \cdots, u_{2k}, z_1, z_2, \cdots, z_{n-2k} \}$ such that 
$$[u_{2i-1}, u_{2i}]= z_1, \forall ~  1 \leq i \leq k; \quad [z_j, L]=0, \forall ~ j \geq 1;$$ $$ [u_i, u_j]=0 \text{ for }
(i,j) \notin \{ (2i-1, 2i), (2i, 2i-1)\}_{1 \leq i \leq k} $$
 Observe that $L'= \mathbb F z_1$. % and $L= Z_2(L).$ 
 Consider two maps $\beta_1, \beta_2: L \longrightarrow L$ as follows,
$$\begin{array}{lllll}
     \beta_1(u_1)=u_1+u_3,& \beta_2(u_2)=u_2 & \beta_1(u_3)=-u_3, &\beta_1(u_4)= -u_4+u_2, &\beta_1(u_j)=u_j, j \geq 5  \\
     \beta_2(u_1)=u_1+u_4,&  \beta_2(u_2)=u_2, &\beta_2(u_3)=-u_1-u_3, &\beta_2(u_4)=-u_4, &\beta_2(u_j)=u_j, j \geq 5 
\end{array}$$
It is easy to check that $\beta_1, \beta_2 \in \mathcal{A}(L).$ But $[u_1, \beta_1\beta_2(u_1)]= z_1 \neq 0$. Thus, $\beta_1\beta_2 \notin \mathcal{A}(L).$ Hence the theorem.
\end{proof} 

\medskip

%%%%%%%%%%%%%%%%%%%%%%%%%%%%%%%%%%%%%%%%%%%%%%%%%%%%%%%%%%%%%%%%%%%%%%%%%%%%%%%%%%%%%%%%%%%%%%%%%%%%%%%%%%%%%%%%%%%%%%%%%%%%%%%%%%%%%%%%%%%%%%%%%%%%%%%%%%%%%%%%%%%%%%%%%%%%%%%%%%%%%%%%%%%%%%%%%%%%%%%%%%%%%%%%%%%%%%%%%%%%%%%%%%%%%%%%%%%%%%%%%%%%%%%%%%%%%%%%%%%%%%%%%%%%%%%%%
\section{Nilpotent Lie algebra of coclass 1}\label{sec_coclass1}

%The following result is one of the main results of this paper.

\begin{theorem} Let $L$ be $n$-dimensional, nilpotent Lie algebra with coclass $1$ over the field $\mathbb F$ of characteristic different from $2$. Then $\mathcal{A}(L)= Aut_c(L)$.   
\end{theorem}

\begin{proof}
Since $L$ is $n$-dimensional nilpotent Lie algebra of maximal class, i.e., $L^{n-1} =0$ and $L^{n-2} \neq 0$. If $\dim L=1,2$, then $L$ is an abelian Lie algebra. Hence, the result follows. Now, we assume that $\dim L \geq 3$. Further $[L: L']=2$ and $[L^{i}: L^{i+1}]=1$ for all $ 1 \leq i \leq n-2$. Hence $Z(L)= L^{n-2}$ and $L^{n-3} \subseteq Z_2(L)$. Choose $u, v \in L$ such that $L=\left\langle u, v \right\rangle$ with $[u,v]:=v_1$ and $[u, v_i]:=v_{i+1} \in L^{i+1}$ for all $1 \leq i \leq n-3.$ We claim that $C_L(u)= \left\langle \ u, Z(L) \ \right\rangle.$ It is easy to check that $C_L(u) \supset \left\langle \ u, Z(L) \ \right\rangle$. For reverse inclusion, let $y \in C_L(u)$, i.e., $y \in L$ such that $[y,u]=0$. Thus $y= b_1 u+ b_2 v+ \sum\limits_{i=1}^{n-2} a_i v_i$. $$\begin{array}{rl}
         0 &= [y,u]  \\
          & = [b_1 u+b_2 v+\sum\limits_{i=1}^{n-2} a_i v_i, u ]\\
          &= b_2 [v,u]+ \sum\limits_{i=1}^{n-2} a_i [v_i, u]\\
          &= b_2 v_1+ \sum\limits_{i=1}^{n-3} a_i v_{i+1}
     \end{array}$$
     But $v_i$'s are linearly independent. Hence $b_2= a_1= \cdots = a_{n-3}=0$. Thus, $y= b_1 u+ a_{n-2}v_{n-2} \in \left\langle \ u, Z(L) \ \right\rangle.$ Hence $C_L(u) \subseteq \left\langle \ u, Z(L) \ \right\rangle$. Hence, the claim follows.\\
 Let $\alpha \in \mathcal{A}(L)$ then $[\alpha(x), x]=0$ for all $x \in L$. By \corref{shush1}, $\alpha(x)-x \in Z_2(L)$ for all $x \in L.$ Since $L= \left\langle \ u, v \ \right\rangle$, we can write
     $\alpha(u)= u+ z_1$ and $\alpha(v)= v+z_2$ where $z_1, z_2 \in Z_2(L)$. 
     $$ [u, z_1]   = [u, \alpha(u)-u]= [u, \alpha(u)]=0, $$ i.e., $z_1 \in C_L(u) =  \left\langle \ u, Z(L) \ \right\rangle$. Therefore, we can write $z_1= k u+ z$ where $k \in \mathbb F$ and $z \in Z(L).$ But $z_1 \in Z_2(L)$ then $$0=[z_1, v_1]= k [u, v_1]+[z,v_1]= k [u, v_1] =- k v_2.$$
    If $k \neq 0$, then $v_2 =0$, i.e., $L^2=0$, which is not true if $\dim L \geq 3$. Hence $k=0$, i.e., $z_1= z \in Z(L)$, implies that $\alpha(u)-u \in Z(L)$. Further,      
$$\begin{array}{rl}
    [z_2, u]= & [\alpha(v)-v, u]\\
    = & [v, \alpha(u)-u] \quad (\text{by } \lemref{firstlem})\\
    =&0 \qquad (\because \alpha(u)-u \in Z(L) )
\end{array}$$
    Thus $z_2 \in C_L(u)$, by using similar arguments as above, we can see that $z_2 \in Z(L)$, i.e., $\alpha(v)-v \in Z(L)$. Hence, $\alpha(x)-x \in Z(L)$ for all $x \in L$., i.e., $\alpha \in Aut_c(L).$ Hence $\mathcal{A}(L) \subseteq Aut_c(L)$. By definition, reverse inclusion holds. Hence, the theorem.    
\end{proof}
\medskip

%%%%%%%%%%%%%%%%%%%%%%%%%%%%%%%%%%%%%%%%%%%%%%%%%%%%%%%%%%%%%%%%%%%%%%%%%%%%%%%%%%%%%%%%%%%%%%%%%%%%%%%%%%%%%%%%%%%%%%%%%%%%%%%%%%%%%%%%%%%%%%%%%%%%%%%%%%%%%%%%%%%%%%%%%%%%%%%%%%%%%%%%%%%%%%%%%%%%%%%%%%%%%%%%%%%%%%%%%%%%%%%%%%%%%%%%%%%%%%%%%%%%%%%%%%%%%%%%%%%%%%%%%%%%%%%%%
\section{Nilpotent Lie algebra of coclass 2}\label{sec_coclass2}

%This subsection will show that $\mathcal{A}(L)$ forms a group if $L$ is nilpotent Lie algebra of coclass $2$.

\begin{theorem}
    Let $L$ be $n$-dimensional, nilpotent Lie algebra with coclass $2$ (nilpotency class $n-2$), over field $\mathbb F$ of characteristic different from $2$. Then $\mathcal{A}(L)$ is a subgroup of $Aut(L)$.
\end{theorem}
\begin{proof} Let $L$ be $n$-dimensional, nilpotent Lie algebra with coclass $2$, i.e. $L^{n-2}=0$ and $L^{n-3} \neq  0.$ If $\dim L=1,2,3$ then $L$ is abelian Lie algebra. Assume $\dim L \geq 4$. Let $\alpha, \beta \in \mathcal{A}(L).$ Note that $L^{n-3} \subseteq Z(L)$ and $L^{n-4} \subseteq Z_2(L).$  Thus $2 \leq \dim Z_2(L) \leq 3$. If $\dim Z_2(L)=2$ then $Z_2(L)=\left\langle \ x, Z(L)\ \right\rangle$  for some $0 \neq x \in Z_2(L) \setminus Z(L)$ as $0 \neq Z(L) \subseteq Z_2(L)$. This implies that $Z_2(L)$ is abelian. By \lemref{abelian}, the result follows in this case. If $\dim Z_2(L)=3$ and $\dim Z(L)=2$ then $[Z_2(L): Z(L)]=1$ implies that $Z_2(L)$ is abelian. By \lemref{abelian}, the result also follows in this case. If $\dim Z_2(L)=3$ and $\dim Z(L)=1$, i.e., $[Z_2(L): Z(L)]=2$ and $Z(L)= L^{n-2}.$ By \lemref{shush2}, result follows.
\end{proof}
\medskip 

%%%%%%%%%%%%%%%%%%%%%%%%%%%%%%%%%%%%%%%%%%%%%%%%%%%%%%%%%%%%%%%%%%%%%%%%%%%%%%%%%%%%%%%%%%%%%%%%%%%%%%%%%%%%%%%%%%%%%%%%%%%%%%%%%%%%%%%%%%%%%%%%%%%%%%%%%%%%%%%%%%%%%%%%%%%%%%%%%%%%%%%%%%%%%%%%%%%%%%%%%%%%%%%%%%%%%%%%%%%%%%%%%%%%%%%%%%%%%%%%%%%%%%%%%%%%%%%%%%%%%%%%%%%%%%%%%
\section{Nilpotent Lie algebra of coclass 3}\label{sec_coclass3}

In coclass $3$, $\mathcal{A}(L)$ is not always a subgroup of $Aut(L)$. The subsequent example serves as a counterinstance, featuring $L$, a $5$-dimensional nilpotent Lie algebra of coclass $3$ with a $1$-dimensional center.

\begin{example}
    Let $L$ be a Lie algebra generated by $$\left\langle\ x_1, x_2, x_3, x_4, x_5: \begin{aligned}
        &[x_1, x_2]= x_5=[x_3,x_4]\\
       & [x_i, x_j]=0 \text{ for }
\{i,j\} \neq \{1,2\}, \{3,4\}    \end{aligned} \right\rangle$$ over field $\mathbb F$. Observe that $Z(L)=L'= \mathbb F x_5$ and $L= Z_2(L)$. Consider two maps $\beta_1, \beta_2: L \longrightarrow L$ as follows,
$$\begin{array}{lllll}
     \beta_1(x_1)=x_3,&  \beta_1(x_2)=x_4, &\beta_1(x_3)=x_1, &\beta_1(x_4)=x_2, &\beta_1(x_5)=x_5 \\
     \beta_2(x_1)=x_1+x_4,&  \beta_2(x_2)=x_2, &\beta_2(x_3)=-x_2-x_3, &\beta_2(x_4)=-x_4, &\beta_2(x_5)=x_5 
\end{array}$$
It is easy to check that $\beta_1, \beta_2 \in \mathcal{A}(L).$ But $[x_1, \beta_1\beta_2(x_1)]= x_5 \neq 0$. Hence, $\beta_1\beta_2 \notin \mathcal{A}(L).$
\end{example}\qed

\begin{theorem}\label{dim_less_5}
    For a nilpotent Lie algebra $L$ with coclass $3$ over field $\mathbb F$ of characteristic different from $2$, if the dimension of $L$ is at most $5$, then the following conditions imply that $\mathcal{A}(L)$ is a subgroup of $Aut(L)$: \begin{itemize}
\item[(i)] $\dim L \leq 4.$ %The dimension of $L$ is less than or equal to 4.
\item[(ii)] $\dim L=5$ and $\dim Z(L) \neq 1.$ %The dimension of $L$ is 5, and the dimension of the center of $L$ is different from 1.
\end{itemize} In other words, if $L$ satisfies either of these conditions, then the set of commuting automorphisms of $L$ form a subgroup of the full automorphism group of $L$.
\end{theorem}

\begin{proof}
    Since $L$ is nilpotent Lie algebra of coclass $3$, i.e, $L^{n-3}=0$. If $\dim L \leq 4$, then $L$ is an abelian Lie algebra. Assume $\dim L=5$. Then, $L^2=0$ and $0 \neq L' \subseteq Z(L)$, $1 \leq \dim L' \leq 3$ as $\dim L \leq 5$. Observe that $2 \leq [L:Z(L)] \leq 4$.
    \begin{itemize}
        \item[Case(a):] If $[L: Z(L)]=2$ then by \lemref{wholeL}, result holds.
         \item[Case(b):] If $[L: Z(L)]=3$ then $\dim L' \leq 2$. Choose $u,v,w \in L$ such that $L= \left\langle u, v,w, Z(L) \right\rangle$. Then $[u,v], [u,w], [v,w] \in L' \subseteq Z(L)$. Let $\alpha \in \mathcal{A}(L)$. Take $$\begin{array}{cc}
             \alpha(u) &= a^\alpha_{11} u +a^\alpha_{12} v +a^\alpha_{13} w+ z_1^\alpha\\
             \alpha(v) &= a_{21}^\alpha u +a_{22}^\alpha v +a_{23}^\alpha w+ z_2^\alpha \\
             \alpha(w) &= a_{31}^\alpha u +a_{32}^\alpha v +a_{33}^\alpha w+ z_3^\alpha  
         \end{array}$$  
         where $a_{ij}^\alpha \in \mathbb F$ for all $1 \leq i,j \leq 3$ and $z_1^\alpha, z_2^\alpha, z_3^\alpha \in Z(L)$. 
         \begin{equation}\label{shush16}\begin{array}{cc}
            0  &= [\alpha(u), u]= -a_{12}^\alpha [u,v]-a_{13}^\alpha [u,w]  \\
            0  & = [\alpha(v), v]= a_{21}^\alpha [u,v]-a_{23}^\alpha [v,w]\\
            0 &= [\alpha(w),w]= a_{31}^\alpha [u,w]+a_{32}^\alpha [v,w]
         \end{array} \end{equation}
        If $\dim L'=2$, we can assume $[v,w]=0$ without loss of generality. By solving system \eqref{shush16}, we get $a_{21}^\alpha=0=a_{31}^\alpha$ and $a_{12}^\alpha=0=a_{13}^\alpha$ as $[u,v]$ and $[u,w]$ are linearly independent vectors. Thus for any  $\alpha \in \mathcal{A}(L)$, 
         $$ \alpha(u) = a^\alpha_{11} u + z_1^\alpha, \quad
             \alpha(v) = a_{22}^\alpha v +a_{23}^\alpha w+ z_2^\alpha, \quad
             \alpha(w) = a_{32}^\alpha v +a_{33}^\alpha w+ z_3^\alpha $$  
         Let $\beta, \gamma \in \mathcal{A}(L)$ then 
         $$ \beta(u) = a^\beta_{11} u + z_1^\beta, \quad
             \beta(v) = a_{22}^\beta v +a_{23}^\beta w+ z_2^\beta, \quad
             \beta(w) = a_{32}^\beta v +a_{33}^\beta w+ z_3^\beta, $$
             $$ \gamma(u) = a^\gamma_{11} u + z_1^\gamma, \quad
             \gamma(v) = a_{22}^\gamma v +a_{23}^\gamma w+ z_2^\gamma, \quad
             \gamma(w) = a_{32}^\gamma v +a_{33}^\gamma w+ z_3^\beta $$ where $a_{ij}^\beta, a_{ij}^{\gamma} \in \mathbb F$  and $z_i^\gamma, z_i^\beta \in Z(L)$ for all $1 \leq i,j \leq 3$.
             Observe that $$[\beta(x), \gamma(x)]=0, \quad \forall x \in \{ u,v,w\}. $$ Hence $[\gamma \beta (x), x]= [\beta(x), \gamma(x)]=0$ for all $x \in L$, i.e., $\gamma \beta \in \mathcal{A}(L)$ for all $\gamma, \beta \in \mathcal{A}(L).$ Hence the result holds in this case. If $\dim L'=1$, we can assume $[v,w]=0=[u,w]$ without loss of generality. We can easily prove the result in this case by doing a similar calculation.
             \item[Case(c):] If $[L: Z(L)]=4$ then $\dim Z(L)=1$ implies that $L'= Z(L)$. Then $L$ is Heisenberg Lie algebra up to isoclinisim. By \thmref{heisenberg}, $\mathcal{A}(L)$ is not a subgroup of $Aut(L).$
    \end{itemize}  
\end{proof}

\begin{corollary}  Let $L$ be $5$-dimensional, nilpotent Lie algebra of coclass $3$ over field $\mathbb F$. $\mathcal{A}(L)$ is not a subgroup of $Aut(L)$ if and only if $L'= Z(L)$ and $\dim Z(L)=1$.    
\end{corollary}

\medskip

\begin{remark}\label{rem_coclass3}
    For $n$-dimensional, $n\geq 6$, nilpotent Lie algebra of coclass $3$, $L^{n-3}=0$ and $0 \neq L^{n-4} \subseteq Z(L)$ and $L^{n-5} \subseteq Z(Z_2(L)) \subseteq Z_2(L).$ Hence $ 1 \leq \dim Z(L) \leq 3$ and $2 \leq \dim Z_2(L) \leq 4.$
\end{remark}

%%%%%%%%%%%%%%%%%%%%%%%%%%%%%%%%%%%%%%%%%%%%%%%%%%%%%%%%%%%%%%%%%%%%%%%%%%%%%%%%%%%%%%%%%%%%%%%%
\begin{lemma}\label{imp_coclass3}
     Let $L$ be $n$-dimensional ($n \geq 6$), nilpotent Lie algebra of coclass $3$ over field $\mathbb F$ of characteristic different from $2$. If one of the following holds:
     \begin{itemize}
         \item[(i)] $\dim Z(L) \neq 1$
         \item[(ii)] $\dim Z(L) = 1$ and $ 2 \leq \dim Z_2(L) \leq 3.$
         \item[(iii)] $\dim Z(L) = 1$, $\dim Z_2(L) =4$ and $ n-3 \leq \dim L' \leq n-2.$
     \end{itemize}
     then $Z_2(L)$ is abelian.
\end{lemma}
\begin{proof}  Let $L$ be $n$-dimensional, $n \geq 6$, nilpotent Lie algebra of coclass $3$ over field $\mathbb F$. By \remref{rem_coclass3},  $ 1 \leq \dim Z(L) \leq 3$.

\textbf{Case (a): If $\dim Z(L) = 3$} then $L/ Z(L)$ is nilpotent Lie algebra of dimension $n-3$ of coclass $1$. Hence $\dim Z \left( L/Z(L) \right) = [Z_2(L): Z(L)]=1$, i.e., $Z_2(L) $ is abelian.

\textbf{Case (b): If $\dim Z(L) = 2$} then $L/ Z(L)$ is nilpotent Lie algebra of dimension $n-2$ of coclass $2$. Thus $ 1 \leq \dim Z \left( L/Z(L) \right) \leq 2$. But $Z_2(L)/Z(L)= Z \left( L/Z(L) \right) $, $[Z_2(L): Z(L)] \leq 2$. If $[Z_2(L): Z(L)]=1$ then $Z_2(L)$ is abelian. So, let $[Z_2(L): Z(L)]= 2$. But $L^{n-5} \subseteq Z(Z_2(L)) \subseteq Z_2(L)$ and $2 \leq \dim L^{n-5} \leq 4$. If $ \dim L^{n-5}=4$ then $Z_2(L)= L^{n-5} = Z(Z_2(L))$, i.e., $Z_2(L)$ is abelian. If $\dim L^{n-5}=3$ then, $[Z_2(L): L^{n-5}] \leq 1$ i.e., $Z_2(L)$ is abelian. If $\dim L^{n-5}=2$ then $[Z_2(L): L^{n-5}] \leq 2$. We can assume, $[Z_2(L): L^{n-5}] =2 $ as otherwise, $Z_2(L) $ is abelian. Now $[Z_2(L): Z(L)]= 2$, $[Z_2(L): L^{n-5}] =2$ and $\dim L^{n-5}=2, \dim Z(L)=2$. We can choose $x \in Z(L)$ such that $x \notin L^{n-5}$. Thus $K_1= \mathbb F x+ L^{n-5} \subseteq Z(Z_2(L))$ is abelian subalgebra of $Z_2(L)$ and $\dim K_1 =3.$ Hence $[Z_2(L): Z(Z_2(L))] \leq 1$, i.e., $Z_2(L)$ is abelian.

\textbf{Case (c): If $\dim Z(L) = 1$} then $[Z_2(L): Z(L)] \leq 3$ as $2 \leq \dim Z_2(L) \leq 4.$
\begin{itemize}
    \item[(1c)] If $[Z_2(L): Z(L)] =1$ then $Z_2(L)$ is abelian.
    \item[(2c)] If $[Z_2(L): Z(L)] =2$, i.e., $\dim Z_2(L)= 3$. But $L^{n-5} \subseteq Z(Z_2(L)) \subseteq Z_2(L)$ implies that $ 2 \leq \dim L^{n-5} \leq 3$. If $\dim L^{n-5} =3$ then $Z_2(L)= L^{n-5} = Z(Z_2(L))$, i.e., $Z_2(L)$ is abelian. If $\dim L^{n-5} =2$, then $[Z_2(L): L^{n-5}] = 1$ i.e., $Z_2(L)$ is abelian.
     \item[(3c)] If $[Z_2(L): Z(L)] =3$, i.e., $\dim Z_2(L)= 4$ then $K_2= L/ Z_2(L)$ is nilpotent Lie algebra of dimension $n-4$ of coclass $1$ implies that $\dim K_2' = n-6.$ Observe that $K_2'= \frac{L'+ Z_2(L)}{Z_2(L)}$ which implies that $\dim (L'+Z_2(L))= n-2$, i.e., $n-4 \leq \dim L' \leq n-2$. But $L' \cap Z_2(L) \subseteq Z(Z_2(L))$.  If $\dim L' = n-2$ then $ \dim L' \cap Z_2(L)= 4 $, i.e., $L' \cap Z_2(L)= Z(Z_2(L))= Z_2(L) $. Hence, $Z_2(L)$ is abelian. If $\dim L'= n-3$ then $\dim L' \cap Z_2(L)=3$, i.e, $[Z_2(L): L' \cap Z_2(L)]=1$ implies that $Z_2(L)$ is abelian.
\end{itemize}

    \noindent 
    Observe that part(i) follows from case(a) and case(b). Part(ii) follows from case(1c) and case(2c). Part(iii) follows from case(3c).
    \end{proof}

\medskip

%%%%%%%%%%%%%%%%%%%%%%%%%%%%%%%%%%%%%%%%%%%%%%%%%%%%%%%%%%%%%%%%%%%%%%%%%%%%%%%%%%%%%%%%%%%%%%%%
The subsequent proposition outlines the structural characteristics of $Z_2(L)$ for finite-dimensional nilpotent Lie algebras with coclass $3$.

\begin{proposition}\label{structure}
    For a nilpotent Lie algebra $L$ with coclass $3$ over field $\mathbb F$ of characteristic different from $2$, if the dimension of $L$ is at least $6$, then the following statements hold: \begin{itemize}
\item[(i)] If the nilpotency class of the second center $Z_2(L)$ is not equal to $2$, then $Z_2(L)$ is abelian.
\item[(ii)] If the nilpotency class of $Z_2(L)$ is $2$, then either $Z_2(L)$ is abelian, or $\dim Z(L)=1, \dim Z_2(L)=4$ and $\dim L'= n-4$.
    \end{itemize}
\end{proposition}

\begin{proof}  Let $L$ be $n$-dimensional, $n \geq 6$, nilpotent Lie algebra of coclass $3$ over field $\mathbb F$. By \remref{rem_coclass3}, $2 \leq \dim Z_2(L) \leq 4$ and $2 \leq \dim L^{n-5} \leq 4$, i.e., $[Z_2(L): L^{n-5}] \leq 2.$

    \begin{itemize}
        \item[case(a):] If $Z_2(L)= L^{n-5}$ then $Z_2(L)$ is abelian as $L^{n-5} \subseteq Z(Z_2(L)) \subseteq Z_2(L)$.
        \item[case(b):] If $[Z_2(L): L^{n-5}]=1$ then $Z_2(L)$ is abelian as $L^{n-5} \subseteq Z(Z_2(L))$.
         \item[case(c):] If $[Z_2(L): L^{n-5}]=2$ then $\dim Z_2(L)=4$ and $\dim L^{n-5}=2$. But $L^{n-5} \subseteq Z(Z_2(L))$, i.e., $\dim Z(Z_2(L)) \geq 2$, implies that nilpotency class of $Z_2(L)$ can be atmost $2$. 
    \end{itemize}
    \begin{itemize}
        \item[(i)] If nilpotency class of $Z_2(L)$ is not $2$ that means $Z_2(L)$ is $1$-step nilpotent. Hence, $Z_2(L)$ is abelian.
        \item[(ii)] If nilpotency class of $Z_2(L)$ is $2$ then either $\dim Z(L) \neq 1$ or $\dim Z(L) =1$. In the first case, $Z_2(L)$ is abelian by \lemref{imp_coclass3}. In the latter case, i.e., $\dim Z(L)=1$, either $\dim Z_2(L) \neq 4$ or $\dim Z_2(L) =4$. If $\dim Z(L)=1$, and $\dim Z_2(L) \neq 4$ then $Z_2(L)$ is abelian by \lemref{imp_coclass3}. Otherwise $\dim Z(L)=1$, and $\dim Z_2(L) = 4$. Either $\dim L' \neq n-4$ or $\dim L'=n-4.$ If $\dim Z(L)=1$,  $\dim Z_2(L) = 4$, and $\dim L' \neq n-4$, i.e., $n-3 \leq \dim L' \leq n-2$ then $Z_2(L)$ is abelian by \lemref{imp_coclass3}.  Otherwise $\dim Z(L)=1$,  $\dim Z_2(L) = 4$, and $\dim L' = n-4.$
         \end{itemize}
\end{proof}

%\medskip
\begin{corollary}
   Let $L$ be $n$-dimensional, $ n \geq 6$, nilpotent Lie algebra with coclass $3$ over field $\mathbb F$ of characteristic not equal $2$. If either nilpotency class of $Z_2(L)$ is not equal to $2$ or $\dim Z(L) \neq 1,$ or $ \dim Z_2(L)\neq 4,$ or $\dim L'\neq n-4$ then $\mathcal{A}(L)$ forms a subgroup of $Aut(L)$. 
\end{corollary}
\begin{proof}
By \propref{structure} and \lemref{imp_coclass3}, in each case, $Z_2(L)$ is abelian. By using \lemref{abelian}, $\mathcal{A}(L)$ forms a subgroup of $Aut(L)$.
\end{proof}

%%%%%%%%%%%%%%%%%%%%%%%%%%%%%%%%%%%%%%%%%%%%%%%%%%%%%%%%%%%%%%%%%%%%%%%%%%%%%%%%%%%%%%%%%%%%%%%%%%%%%%%%%%%%%%%%%%%%%%%%%%%%%%%%%%%%%%%%%%%%%%%%%%%%%%%%%%%%%%%%%%%%%%%%%%%%%%%%%%%%%%%%%%%%%%%%%%%%%%%%%%%%%%%%%%%%%%%%%%%%%%%%%%%%%%%%%%%%%%%%%%%%%%%%%%%%%%%%%%%%%%%%%%%%%%%%%
\bibliography{ref}
\bibliographystyle{alpha}

%%%%%%%

%%%%%%%%%%%%%%%%%%%%%%%%%
\end{document}